\documentclass{amsart}
\usepackage{amssymb}
\usepackage{amsmath}
\usepackage{amsthm}
\usepackage{enumerate}
\usepackage{graphicx}
\usepackage{color}

\def\l{\lambda}

\def\R{{\mathbb R}}

\def\A{\mathcal{A}}

\newtheorem{theorem}{Theorem}[section]
\newtheorem{lemma}[theorem]{Lemma}

\newtheorem{definition}[theorem]{Definition}

\newtheorem{remark}{\bf Remark}[section]

\newfont{\Bb}{msbm10 scaled\magstep{1}}

\begin{document}

\title[Cracked beams and  arches]{Variational setting for cracked beams and shallow arches}

\author{Semion Gutman, Junhong Ha and Sudeok Shon}
\address{$^{1}$ Department of Mathematics, University of Oklahoma, Norman, Oklahoma 73019, USA, e-mail: sgutman@ou.edu}
\address{$^{1}$School of Liberal Arts, Korea University of Technology and Education,
        Cheonan 31253, South Korea, e-mail: hjh@koreatech.ac.kr}
\address{$^{3}$Department of Architectural Engineering, Korea University of Technology and Education,
        Cheonan 31253, South Korea, e-mail: sdshon@koreatech.ac.kr}



\begin{abstract} 
We develop a rigorous mathematical framework for the weak formulation of cracked beams and shallow arches problems. First, we discuss the crack modeling by means of massless rotational springs. Then we introduce Hilbert spaces, which are sufficiently wide to accommodate such representations. Our main result is the introduction of a specially designed linear operator that "absorbs" the boundary conditions at the cracks.

 We also provide mathematical justification and derivation of the Modified Shifrin's method for an efficient computation of the eigenvalues and the eigenfunctions for cracked beams.  
\end{abstract}

\maketitle


\section{Introduction}\label{section:intro}
Detection of cracks is an important engineering problem. It requires a rigorous mathematical framework  for modeling the motion of cracked beams and shallow arches. 
The main goal of this paper is to develop a variational setting for such a framework.
 We also present an elegant and efficient method (a modification of the Shifrin's method) for the computation of the eigenvalues and the eigenfunctions for the cracked elements. 

For a theory of cracked Bernoulli-Euler beams see  \cite{Christides1984}.  
A significant effort has been directed at the vibration analysis of cracked beams. Representation of a crack by a rotational spring has been proven to be accurate, and it is often used, see  \cite{CADDEMI2009, CADDEMI2013944} and the extensive bibliography there. Determination of the beam natural frequencies is discussed in  \cite{LIN2002987, OSTACHOWICZ1991191, SHIFRIN1999}. 
S. Caddemi and his colleagues have further developed the theory using energy functions in \cite{CADDEMI2013944}. However, no full variational setting has been presented so far, making it difficult to study evolution problems. As we have already mentioned, our work closes a gap in this development.

The theory of uniform beams and shallow arches is well developed. An early exposition can be found in  \cite{Ball1(1973)}. More general models in the multidimensional setting, and a literature survey are presented in \cite{Emmrich_2011}. A review for vibrating beams is given in \cite{HAN1999}. Motion of uniform arches and a related parameter estimation problem are studied in \cite{GUTMAN2013297}. These results are extended to point loads in \cite{GH2017}. The existence of a compact, uniform attractor is established in \cite{GUTMAN2018557}. 

 The transverse motion of a beam or an  arch is described by the function $y(x,t),\, x\in [0,\pi],\, t\geq 0$, which represents the deformation of the beam/arch measured from the $x$-axis. For definiteness, the boundary conditions are of the hinged type
 \begin{equation}\label{intro:eq32}
y(0,t)=y''(0,t)=0,\quad y(\pi,t)=y''(\pi,t)=0,\quad t \in (0,T).
\end{equation}
 Other types of boundary conditions, can be treated similarly.
 
 Crack modeling is considered in Section \ref{section:crack}.
Suppose that there are $m$ cracks located at $0<x_1<x_2<\dots<x_m<\pi$.
A crack at $x=x_i$ is represented by a rotational spring with the flexibility $\theta_i$, $i=1,\dots,m$. This is expressed as 
\begin{equation}\label{intro:eq34}
y'(x_i^+,t)-y'(x_i^-,t)=\theta_i y''(x_i,t),\quad t>0,\ i=1,\dots, m.
\end{equation}
In Section \ref{section:hilbert} we introduce special Hilbert spaces $V$ and $H$ satisfying
\begin{equation}\label{intro:eq36}
V\subset  H\subset V',
\end{equation}
with continuous and dense embeddings. These spaces are broad enough to contain continuous functions with discontinuous derivatives at the joint points.

Section \ref{section:var} contains our main result.
We introduce the operator $\A : V\to V'$, by
\begin{equation}\label{intro:eq38}
\langle \A u,v\rangle_V= \sum_{i=1}^{m+1} (u'', v'')_i+\sum_{i=1}^m \frac 1\theta_i J[u'](x_i) J[v'](x_i), 
\end{equation}
for any $u,v\in V$, where $J[u'](x)=u'(x^+)-u'(x^-)$.

Then we show that the solution $u$ of the equation $\A u=f$ in $H$ satisfies the joint conditions, including \eqref{intro:eq34}. Thus the operator $\A$ "absorbs" the boundary conditions, as expected of the weak  formulation of the steady state problem.

This result allows us to prove the existence of the eigenvalues and the eigenfunctions of $\A$.
An efficient Modified Shifrin's Method for their computation is presented in Section \ref{section:eig}.

The results in this paper form the basis for a comprehensive study of dynamic behavior of cracked beams and arches. It will be presented elsewhere.

\section{Crack modeling}\label{section:crack}
\setcounter{equation}{0}
  A crack is a  disruption in the material, that has a negligible extent in the direction of the beam/arch axis, but of a non-negligible depth. It is fully described by its position along the axis, and the crack depth ratio $\hat\mu$, as shown in Figure \ref{figCracks}. 
\begin{figure}
\begin{center}
\includegraphics[height=3.5cm,width=11.5cm]{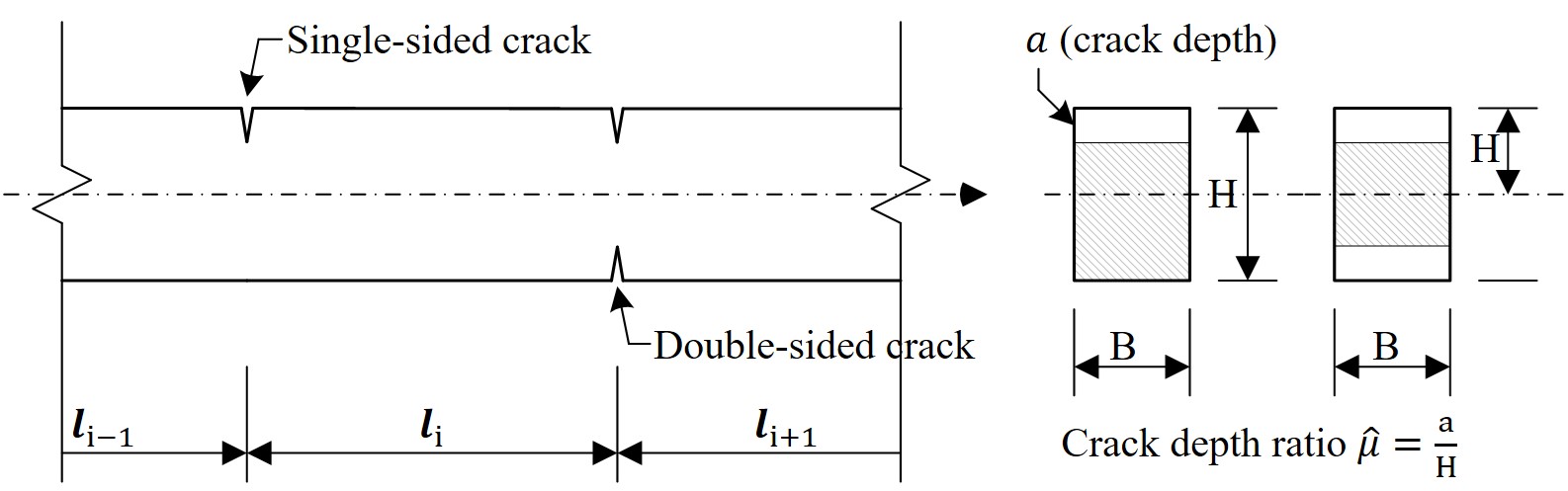}
\end{center}
\vspace{-0.5cm}
\caption{Crack parameters.} \label{figCracks}
\end{figure}

According to \cite{CERRI200439}, a crack is modeled by a massless rotational spring.
The  spring flexibility $\theta=\theta(\hat\mu)$ depends on the crack depth ratio $\hat\mu$, and on whether the crack is one-sided or two-sided, open or closed, and so on. The flexibility $\theta$ is equal to $0$ if there is no crack, and it increases with the crack depth.
Explicit expressions for the functions $\theta(\hat\mu)$ are provided in Section \ref{section:eig}.  

\begin{remark}
The following discussion is applicable to both arches and beams, but to avoid repetitions we will refer just to arches. 
\end{remark}

Suppose that there are $m$ cracks along the length of the arch, located at $0<x_1<\dots <x_m<\pi$.  For convenience, we denote $x_0=0$, and  $x_{m+1}=\pi$. 
Consequently, the cracked arch is modeled as a collection of $m+1$ uniform arches over the intervals $l_i=(x_{i-1},x_i),\, i=1,\dots, m+1$, as shown in Figure \ref{figCrArch}(b).

We consider only the transverse motion of the arch, so its position can be described by the function $y=y(x,t)$, $0\leq x\leq \pi$, $t\geq 0$.
The boundary conditions at the cracks enforce the continuity of the displacement field $y$, the bending moment $y''$, and the shear force $y'''$. 
Condition $y'(x_i^+,t)-y'(x_i^-,t)=\theta_i y''(x_i^+,t)$ expresses the discontinuity of the arch slope at the $i$-th crack, where $\theta_i=\theta(\hat\mu_i)$, see Figure \ref{figCrArch}(b). 

\begin{figure}[b]
\begin{center}
\includegraphics[height=2.5cm,width=11.5cm]{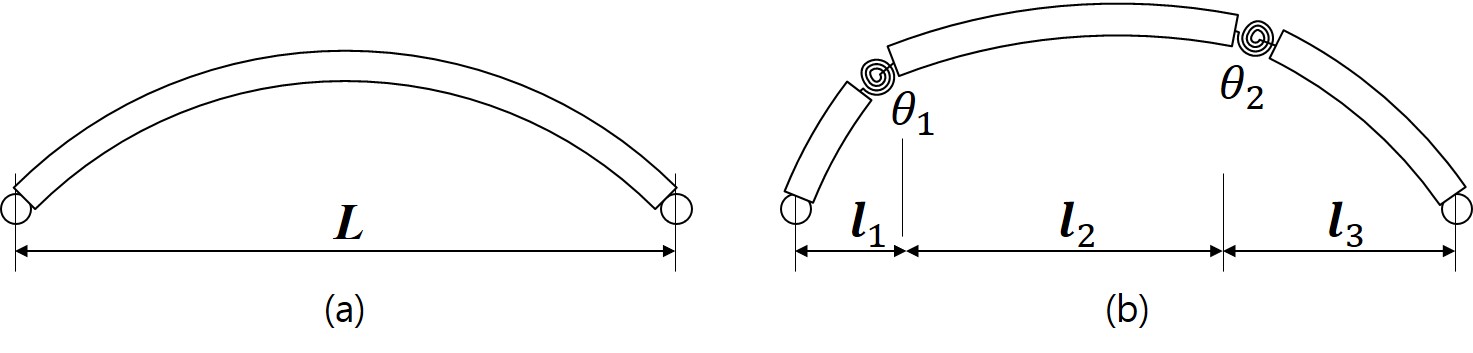}
\end{center}
\vspace{-0.5cm}
\caption{Beam or shallow arch: (a) uniform, (b) with two cracks.} \label{figCrArch}
\end{figure}

To simplify the statement of the boundary conditions at the cracks, we introduce the notion of the \emph{jump} $J[u](x)$ of a function $u=u(x)$ at any $x\in(0,\pi)$, as follows
\begin{equation}\label{intro:eq16}
 J[u](x)=u(x^+)-u(x^-).
\end{equation}

With this notation the conditions at the cracks (joint conditions) are
\begin{equation}\label{intro:eq18}
 J[y](x_i,t)=0,\quad J[y''](x_i,t)=0,\quad J[y'''](x_i,t)=0,
\end{equation}
and
\begin{equation}\label{intro:eq20}
 J[y'](x_i,t)=\theta_i y''(x_i^+,t),
\end{equation}
where $\theta_i=\theta(\hat\mu_i)$, $i=1,2,\dots,m$ and $t\geq 0$. Note that $y''(x_i^+,t)=y''(x_i^-,t)$ by \eqref{intro:eq18}.

\section{Hilbert spaces}\label{section:hilbert}
\setcounter{equation}{0}
We introduce Hilbert spaces $H$ and $V$ suitable for working with cracked elements.
Suppose that an arch has $m$ cracks at the joint points $0<x_1<\dots <x_m<\pi$.
 This partition of the interval $[0,\pi]$ is associated with $m+1$ subintervals $l_i=(x_{i-1},x_i),\, m=1,\dots, m+1$. 
 
Let $H$ be the Hilbert space
\begin{equation}\label{hilbert:eq10}
H=\bigoplus_{i=1}^{m+1} L^2(l_i).
\end{equation}

Let the inner product and the norm in $L^2(l_i)$ be denoted by $(\cdot,\cdot)_i$ and $|\cdot|_i$ correspondingly. 
The inner product and the norm in $H$ are defined by
\begin{equation}\label{hilbert:eq12}
(u,v)_H=\sum_{i=1}^{m+1}(u,v)_i, \quad |u|^2_H= \sum_{i=1}^{m+1}|u|^2_i.
\end{equation}

Consider the Sobolev space $H^2(a,b)$ on a bounded interval $(a,b)\subset \R$, and let
$u\in H^2(a,b)$. Then $u, u'$ are continuous functions on $[a,b]$, up to a set of measure zero, and $u''\in L^2(a,b)$. Therefore, for such $u$, we will always assume that $u, u'\in C[a,b]$.

Define the linear space
\begin{equation}\label{hilbert:eq22}
V=\left\{ u\in \bigoplus_{i=1}^{m+1} H^2(l_i)\, :\, u(0)=u(\pi)=0,\ J[u](x_i)=0,\ i=1,\dots, m\right\}.
\end{equation}
We interpret  $u\in V$ as a continuous function on $[0,\pi]$, such that $u(0)=u(\pi)=0$, with $u'\in L^2(0,\pi)$, i.e. $u\in H_0^1(0,\pi)$. Furthermore,  $u|_{l_i}\in H^2(l_i)$, and $u'|_{l_i}\in C[x_{i-1}, x_i]$ for  $i=1,2,\dots, m+1$. 

 Define the inner product on $V$  by
\begin{equation}\label{hilbert:eq24}
((u,v))_V=\sum_{i=1}^{m+1} (u'', v'')_i+\sum_{i=1}^m J[u'](x_i) J[v'](x_i), \quad\text{for any}\ u,v\in V,
\end{equation}
where $(u'',v'')_i=\int_{l_i} u''(x) v''(x)\, dx$.

It is clear that $((\cdot,\cdot))_V$ is a symmetric, bilinear form on $V$. To see that $((u,u))_V=0$ implies  $u=0$, notice that any function $u$ with $((u,u))_V=0$ is piecewise linear and continuous on $[0,\pi]$. Furthermore, $J[u'](x_i)=0$ for any $i=1,2,\dots,m$. Therefore $u'$ is continuous on $[0,\pi]$. In fact, it is a constant there, since $u''=0$ a.e. on $[0,\pi]$. Thus $u$ is a linear function on $[0,\pi]$ satisfying the zero boundary conditions at the ends of the interval. Therefore $u=0$ on $[0,\pi]$, and $((\cdot,\cdot))_V$ is a well-defined inner product on $V$. The corresponding norm in $V$ is
\begin{equation}\label{hilbert:eq26}
\|u\|_V^2=\sum_{i=1}^{m+1} |u''|^2_i+\sum_{i=1}^m |J[u'](x_i)|^2, \quad\text{for any}\ u\in V,
\end{equation}
where $|\cdot|_i$ is the norm in $L^2(l_i)$. We will show in Lemma \ref{hilbert:lemma16} that $V$ is a Hilbert space. 

Let $u\in V$. We define the derivatives of $u$ component-wise in the spaces $H^2(l_i)$, that is $u'(x)=(u|_{l_i})'(x), u''(x)=(u|_{l_i})''(x)$, and so on, for $x\in l_i$, $i=1,\dots,m+1$. For definiteness, we will assume that the derivative $u'$ is continuous from the right on $[0,\pi]$.

Some useful properties of functions in $V$ are established in the following lemma.
\begin{lemma}\label{hilbert:lemma12}
Let $c\geq 0$ denote various constants independent of $u\in V$. Then
\begin{enumerate}[(i)]
\item The second derivative $u''$ is bounded in $H$, and
\begin{equation}\label{hilbert:eq40}
|u''|_H\leq \|u\|_V.
\end{equation}
\item The derivative $u'$ is bounded on $[0,\pi]$,
\begin{equation}\label{hilbert:eq42}
\sup \{|u'(x)|,\ x\in [0,\pi]\}\leq c\left(|u'|_H+  |u''|_H \right).   
\end{equation}
Moreover,
\begin{equation}\label{hilbert:eq43}
|u'|_H\leq c\|u\|_V,\quad\text{and}\quad  \sup\{|u'(x)|,\ x\in [0,\pi]\}\leq \|u\|_V.
\end{equation}
\item Function $u$ is Lipschitz continuous, with the Lipschitz constant $c\|u\|_V$. Also, $u$ is bounded on $[0,\pi]$, 
\begin{equation}\label{hilbert:eq44}
\|u\|_\infty=\max \{|u(x)|,\ x\in [0,\pi]\}\leq c\|u\|_V,
\end{equation}
and
\begin{equation}\label{hilbert:eq47}
|u|_H\leq c\|u\|_V.
\end{equation}

\end{enumerate}
\end{lemma}

\begin{proof} We only show {\it (ii)}.
Let $u\in V$. Then
its derivative $u'$ is continuous on any interval $[x_{i-1},x_i]$, $i=1,\dots, m+1$. By the Mean Value Theorem for Integrals, there exists $c_i\in [x_{i-1},x_i]$, such that
\[
u'(c_i)=\frac 1 {|l_i|}\int_{l_i} u'(s)\, ds.
\]
Thus $|u'(c_i)|\leq c|u'|_H$. Also, for any $x\in [x_{i-1},x_i]$,
\[
|u'(x)-u'(c_i)|\leq\int_{l_i} |u''(s)|\, ds\leq c|u''|_H.
\]
Therefore 
$|u'(x)|\leq c\left(|u'|_H+  |u''|_H \right)$
for any $x\in[0,\pi]$, giving \eqref{hilbert:eq42}. This inequality implies
 $|J[u'](x)|\leq 2c\left(|u'|_H+  |u''|_H \right)$.

We have
\begin{align}\label{hilbert:eq50}
 \int_a^{b} u''(x)\, dx&=u'|_a^b+\sum_{a< x_i\leq b}\left(u'(x_i^-)-u'(x_i^+)\right)\\
 &= u'(b)-u'(a)-\sum_{a< x_i\leq b} J[u'](x_i).\nonumber
\end{align}
Therefore
\begin{equation}\label{hilbert:eq52}
|u'(b)-u'(a)|\leq \int_0^L|u''(x)|\, dx+\sum_{i=1}^m |J[u'](x_i) |\leq c\|u\|_V.
\end{equation}

First, choose $a\in[0,\pi]$ be such that $u'(a)\leq 0$, which is always possible, since $u(0)=u(L)=0$. Then, by  \eqref{hilbert:eq52}, for any $b\in [0,\pi]$ we have  $u'(b)\leq c\|u\|_V$. This establishes the upper bound for $u'(x),\, x\in [0,\pi]$. Similarly, choosing $a$ such that $u'(a)\geq 0$, we establish the lower bound for $u'(x),\, x\in [0,\pi]$. Inequalities in \eqref{hilbert:eq43} follow.
\end{proof}

\begin{lemma}\label{hilbert:lemma16}
Let $u\in V$, 
\[
N^2_1(u)=|u|^2_H+|u'|^2_H+|u''|^2_H,
\]
and
\[
N^2_2(u)=|u'|^2_H+|u''|^2_H.
\]
Then 
\begin{enumerate}[(i)]
\item The norms $N_1, N_2$ and $\|\cdot\|_V$ are equivalent on $V$.
\item $\bigoplus_{i=1}^{m+1} H^2(l_i)$ is a Hilbert space, and $V$ is its closed subspace of \newline co-dimension $m+2$. 
\item $V$ is a Hilbert space.
\end{enumerate}
\end{lemma}
\begin{proof} 
{\it (i)} This follows from Lemma  \ref{hilbert:lemma12}, and the observation that $|J[u'](x)|\leq 2c\left(|u'|_H+  |u''|_H \right)$. 

{\it (ii)}
Let $X=\bigoplus_{i=1}^{m+1} H^2(l_i)$. Then $X$ is a Hilbert space with the norm $N_1(\cdot)$. By the Trace Theorem \cite[Theorem 3.2]{LionsMagenes}, functionals $g_0(u)=u(0^+),$ and $g_\pi(u)=u(\pi^-)$, $u\in X$ are continuous linear functional on $X$. Therefore $\{u\in X : g_0(u)=0,\ g_\pi(u)=0\}$ is a closed subspace of $X$ of co-dimension $2$ in $X$. 

Similarly, the functionals $J[u](x_i)=u(x_i^+)-u(x_i^-)$, $u\in X$, $i=1,\dots, m$, are linear and continuous on $X$. By the definition \eqref{hilbert:eq22} of $V$, we conclude that $(V, N_1(\cdot))$ is a closed subspace of $X$  of co-dimension $m+2$.

It remains to show that on $V$ the norm $N_1$ is equivalent to the norm $\|\cdot\|_V$, defined in \eqref{hilbert:eq26}, but this was established in {\it (i)}.

{\it (iii)}
 Since the space $X=\bigoplus_{i=1}^{m+1} H^2(l_i)$ is complete, then so is  its closed subspace $V$.
\end{proof}
%
%
 
 \begin{lemma}\label{hilbert:lemma2} 
The identity embedding $i : V\to H$ is linear, continuous, with a dense range in $H_0^1$. Furthermore, it is compact.
\end{lemma}
\begin{proof}
The embedding is linear. By \eqref{hilbert:eq47}, $|u|_H\leq c\|u\|_V$ for $u\in V$. Thus the embedding is continuous. 

Let $B\subset V$ be the unit ball of $V$. By Lemma \ref{hilbert:lemma12},
 functions $u\in B$ are equicontinuous and equibounded on $[0,\pi]$. Hence they form a precompact set in $C[0,\pi]$, and in $L^2(0,\pi)$. Similarly, functions $\{u'\}_{u\in B}$ are precompact in $C[x_{i-1},x_i]$, for any $i=1,\dots,m+1$, hence they are precompact in $L^2(0,\pi)$.  The compactness of the embedding follows.

For the density of the embedding, note that $C_0^\infty(0,\pi)\subset V\subset H$, and $C_0^\infty(0,\pi)$ is dense in $H$. Therefore $V$ is dense in $H$.
\end{proof}

Lemma \ref{hilbert:lemma2} allows us to define the Gelfand triple $V\subset H\subset V'$, with 
the dense embeddings. Furthermore, the embedding $ V\subset H$ is compact.
The pairing $\langle\cdot,\cdot\rangle_V$ between $V$ and $V'$ extending the inner product in $H$. This means that given $f\in H=H'\subset V'$, and $v\in V$, we have $\langle f,v\rangle_V=(f,v)_H$.

\setcounter{equation}{0}
\section{Variational setting and operator $\A$}\label{section:var}

We introduce the operator $\A : V\to V'$ that "absorbs" the junction boundary conditions.
This operator is central to the variational setting of problems for cracked beams and arches.
The existence of its eigenvalues and the eigenfunctions is established as well.

\begin{definition}\label{var:def2}
Define the  operator $\A$ on $V$ by
\begin{equation}\label{var:eq8}
\langle \A u,v\rangle_V= \sum_{i=1}^{m+1} (u'', v'')_i+\sum_{i=1}^m \frac 1\theta_i J[u'](x_i) J[v'](x_i), 
\end{equation}
for any $u,v\in V$. We will also write $\langle \A u,v\rangle$ for $\langle \A u,v\rangle_V$, if it does not cause a confusion.
\end{definition}
See Section \ref{section:crack} for the setup for the junction (crack) points $x_i$, and the flexibilities $\theta_i$. Recall, that a linear operator $A : V\to V'$ is called coercive, if there exists $c>0$, such that $\langle Au, u\rangle\geq c\|u\|^2_V$ for any $u\in V$.

\begin{lemma}\label{var:lemma1}
Let $\A$ be defined by \eqref{var:eq8}. Then $\A$ is a symmetric, continuous, linear, and coercive operator from $V$ onto $V'$.
\end{lemma}
\begin{proof}
Clearly, $\A$ is a symmetric linear operator. Since all $\theta_i>0$, we conclude that there exists a constant $C>0$, such that $|\langle \A u,v\rangle|\leq C\|u\|_V\|v\|_V$. Therefore $\A$ is defined on all of $V$, and it is bounded.

Similarly,
\[
|\langle \A u,u\rangle|=\sum_{i=1}^{m+1} |u''|^2+\sum_{i=1}^m \frac 1\theta_i |J[u'](x_i)|^2\geq c\|u\|_V^2.
\]
Therefore $\A$ is coercive on $V$, and its range is $V'$, see \cite[Theorem 2.2.1]{temam1997infinite}.
\end{proof}

As was mentioned in Section \ref{section:crack}, functions $u=u(x)$ modeling an arch with cracks are expected to satisfy certain boundary conditions. For convenience, we restate them here:
\begin{equation}\label{var:eq2}
u(0)=u(\pi)=0,\quad u''(0)=u''(\pi)=0,
\end{equation}
and 
\begin{equation}\label{var:eq4}
 J[u](x_i)=0,\quad J[u''](x_i)=0,\quad J[u'''](x_i)=0,\quad J[u'](x_i)=\theta_i u''(x_i^+),
\end{equation}
for $i=1,\dots,m$.

The next theorem is the main result of this paper.
\begin{theorem}\label{var:thm2}
Let the domain of $\A$ be $D(\A)=\{ v\in V : \A v\in H\}$.
\begin{enumerate}[(i)]
\item If $u\in D(\A)$, then $u|_{l_i}\in H^4(l_i)$, $\A u=u''''$ a.e. on $l_i$,  $i=1,\dots,m+1$, and $u$ satisfies  conditions \eqref{var:eq2}--\eqref{var:eq4}. 
\item Let $f\in H$, then equation $\A u=f$ in $V'$ has a unique solution $u\in D(\A)$.
\end{enumerate}
\end{theorem}
\begin{proof}
By Lemma \ref{var:lemma1}, the operator $\A$ is coercive, and its range is $V'$. Since $H=H'\subset V'$, condition $f\in H$ implies that $f\in V'$. Therefore equation $\A u=f$ in $V'$ has a solution $u\in D(\A)$, which is unique since $\A$ is coercive. 

To investigate the properties of functions in $D(\A)$, recall that $l_1=(x_0, x_1)$. Notice that $C_0^\infty(l_1)\subset V$, where it is assumed that the  functions from  $C_0^\infty(l_1)$ are extended by zero outside of $l_1$. Thus $v(x)=v'(x)=0$, for $x=0$ and any $x\geq x_1$, $v\in C_0^\infty(l_1)$.

Let $u\in D(\A)$, so $\A u=f$ for some $f\in H$. By the definition of $V$, we have $u|_{l_1}\in H^2(l_1)$.
For any $v\in C_0^\infty(l_1)$, by the definition of $\A$, we have
\[
\langle \A u,v\rangle=\int_{l_1} u''(x)v''(x)\, dx.
\]
Integration by parts gives
\[
\int_{l_1} u''(x)v''(x)\, dx=\int_{l_1} u(x)v''''(x)\, dx=(f,v)_H=\int_{l_1} f(x) v(x)\, dx.
\]
Therefore $D^{(4)}u=f$ in the sense of the weak derivatives on $l_1$.
Thus $u|_{l_1}\in H^4(l_1)$, and $u''''=f$ a.e. on $l_1$.
Repeating this argument for other intervals $l_i$, we conclude that 
$u|_{l_i}\in H^4(l_i)$, and $u''''=f$ a.e. on $l_i$, $i=1,\dots,m+1$. 

It remains to show the satisfaction of the conditions \eqref{var:eq2}--\eqref{var:eq4}.
So, let $\A u=f\in H$. Since we have already established that $u|_{l_i}\in H^4(l_i)$, $i=1,\dots,m+1$, we can do the Integration by Parts on every interval $l_i$, to obtain that for \emph{any}  $v\in V$
\begin{align*}
(f,v)_H&=\langle\A u,v\rangle=\sum_{i=1}^{m+1}( u'', v'')_i+\sum_{i=1}^m \frac 1\theta_i J[u'](x_i) J[v'](x_i)\\
&=\sum_{i=1}^{m+1}( u'''', v)_i - u''' v|_0^\pi+\sum_{i=1}^m u''' v|_{x_i^-}^{x_i^+}+u''v'|_0^\pi-\sum_{i=1}^m u'' v'|_{x_i^-}^{x_i^+}\nonumber\\
& +\sum_{i=1}^m \frac 1\theta_i J[u'](x_i) J[v'](x_i).\nonumber 
\end{align*}
 Since $v\in V$, we have $v(0)=v(\pi)=0$, and $v$ is continuous on $[0,\pi]$. Therefore the above equality can be rewritten as
\begin{align}\label{var:eq62}
&\sum_{i=1}^{m+1}( u''''-f, v)_i +\sum_{i=1}^m J[u'''](x_i) v(x_i)+u''v'|_0^\pi-\sum_{i=1}^m u'' v'|_{x_i^-}^{x_i^+}\\
& +\sum_{i=1}^m \frac 1\theta_i J[u'](x_i) J[v'](x_i)=0.\nonumber 
\end{align}
The first sum is zero, since $u''''=f$ a.e. on $l_i$, $i=1,\dots,m+1$.
 Next, choose a continuously differentiable $v\in V$, which is not zero only in a small neighborhood of $x=0$, and $v'(0)\not=0$. Conclude that $u''(0)=0$. Similarly, $u''(\pi)=0$. 
 
 Choose a continuously differentiable $v\in V$, such that $v'(x_i)=0$, and $v(x_i)=0$ for all $i=2,\dots,m$, but $v(x_1)\not=0$, $v'(x_1)=0$. Conclude that $J[u'''](x_1)=0$. Repeat this procedure for other points $x_i$, one at a time. Thus $J[u'''](x_i)=0$, $i=1,\dots,m$.  We are left with
\begin{equation}\label{var:eq64}
\sum_{i=1}^m\left[\frac 1\theta_i J[u'](x_i) J[v'](x_i)- u'' v'|_{x_i^-}^{x_i^+} \right]=0.
\end{equation}
Choose a continuously differentiable $v\in V$, which is not zero only in a small neighborhood of $x_1$, and such that $v'(x_1)\not=0$. This implies $J[u''](x_1)v'(x_1)=0$. Therefore $J[u''](x_1)=0$. Repeat for other points $x_i$. Thus $u''(x_i^+)=u''(x_i^-)$ for $i=1,\dots,m$. Now we can rewrite \eqref{var:eq64} as
\[
\sum_{i=1}^m\left[\frac 1\theta_i J[u'](x_i)-u''(x_i^+) \right] J[v'](x_i)=0.
\]
Choose a continuous, piecewise linear $v\in V$, such that  $v(x_1)=1$,  $v$ is linear on $[0,x_1]$, and on $[x_1,\pi]$. Note that $J[v'](x_i)=0$ for $i=2,\dots,m$, and $J[v'](x_1)\not=0$. Conclude that $J[u'](x_1)=\theta_1 u''(x_1^+)$. Repeat for other points $x_i$, $i=2,\dots,m$. Thus $u$  satisfies all the conditions \eqref{var:eq2}--\eqref{var:eq4}.
\end{proof}

{\bf Remark}. The fact that $u''''=f$ a.e. on $(0,\pi)$ in Theorem \ref{var:thm2} does not imply that $u\in H^4(0,\pi)$. This is similar to the fact that the strong derivative $p'$ of a step function $p$ on $(0,\pi)$ is zero a.e. on $(0,\pi)$. However,  $p\not\in H^1(0,\pi)$. 

Finally in this section we discuss the {\bf eigenvalues} and the {\bf eigenfunctions} of the operator $\A$.
 It was shown in Lemma 
\ref{var:lemma1} that $\A$ is a continuous, linear, symmetric, and coercive operator from $V$ onto $V'$. 
Following \cite[Section 2.2.1]{temam1997infinite}, 
$\A$ can also be considered as an unbounded operator in $H$. By Lemma \ref{hilbert:lemma2}, the embedding $V\subset H$ is compact. Therefore the standard spectral theory for Sturm-Liouville boundary value problems is applicable. The eigenfunctions belong to $H$. Therefore, by Theorem \ref{var:thm2}, they are in the domain  $D(\A)\subset V$, thus continuous on $[0,\pi]$, and satisfy  conditions \eqref{var:eq2}--\eqref{var:eq4}.

 We summarize these results in the following lemma.
\begin{lemma}\label{var:lemma20}
Let $\A$ be the operator defined in \eqref{var:eq8}. Then
\begin{enumerate}[(i)]
\item There exists an increasing sequence of its real positive eigenvalues \newline
$\l^4_1,\l^4_2,\dots$, with $\lim_{k\to\infty}\l^4_k=\infty$. 
\item The corresponding eigenfunctions $\varphi_k\in D(\A)\subset V$, $k\geq 1$, and they satisfy the junction conditions \eqref{var:eq2}--\eqref{var:eq4}.
\item The eigenfunctions $\varphi_k$ satisfy $\A\varphi_k=\l^4_k\varphi_k$ in $H$, $k \geq 1$. That is, $\varphi_k''''(x)=\l^4_k\varphi_k(x)$ a.e. on every interval $l_i$, $i=1,\dots, m+1$. 
\item The set $\{\varphi_k\}_{k=1}^\infty$ is a complete orthonormal basis in $H$.

\end{enumerate} 
\end{lemma}

Algorithms for a computational determination of the eigenvalues and the eigenfunctions of $\A$ are discussed in Section \ref{section:eig}.

{\bf Remark.} If the arch is uniform, i.e. it has no cracks, then  the results presented in this section are simplified. Specifically, the spaces $V, H$, and the operator $\A$ take the following forms
\begin{equation}\label{var:eq70}
V=H^2(0,\pi)\cap H_0^1(0,\pi),\quad H=L^2(0,\pi),\quad \langle\A u,v\rangle_V=(u'',v'')_H,
\end{equation}
for any $u,v\in V$.
See \cite{GUTMAN2013297} for an investigation of this case.

\setcounter{equation}{0}
\section{Eigenvalues and eigenfunctions}\label{section:eig}
In this section we present the Modified Shifrin's method for the computation of the eigenfunctions $\varphi_k,\, k\geq 1$ and the corresponding eigenvalues $\l^4_k$, of the operator $\A$.  The existence of the eigenvalues and the eigenfunctions has been established in Lemma \ref{var:lemma20}.

\vskip2mm
{\bf Transition matrices method}. This is a common method for the determination of the eigenfunctions and the eigenvalues, so we just briefly mention it for  completeness, see \cite{LIN2002987} for details.

The general solution of the equation $w''''=\l^4 w$ on $l_i=(x_{i-1}, x_i)$, for $i=1,\dots,m+1$ is 
\begin{align}\label{eig:eq4}
 w_i^\lambda(x) &= A_i \sin \lambda (x-x_{i-1}) + B_i \cos \lambda(x-x_{i-1}) \\ 
  & + C_i \sinh \lambda (x-x_{i-1}) + D_i \cosh \lambda (x-x_{i-1}).\nonumber
\end{align}
Let  vector $\vec{A}_i=[A_i,B_i,C_i,D_i]^T$ be composed of the coefficients of the expansion in \eqref{eig:eq4}, on interval $l_i$. 

Suppose that vector $\vec A_1$ is known. Then it defines function $w_1^\l$ on $l_1$, and the boundary conditions for $w_1^\l$ at $x_1^-$. The junction conditions \eqref{var:eq2}--\eqref{var:eq4} define the boundary conditions for $w_2^\l$ at $x_1^+$. Then the initial value problem $(w_2^\l)''''=\l^4 w_2^\l$ on $l_2=(x_1, x_2)$, uniquely defines the expansion coefficients $\vec A_2$ on $l_2$. It is readily seen that the transformation from $\vec A_1$ to $\vec A_2$ is linear, and it is given by a $4\times 4$ matrix $T^{(1)}$, i.e. $\vec A_2=T^{(1)}\vec A_1$. 

Extending this process to all the subintervals $l_i, i=1,\dots, m+1$, we get
\begin{equation}\label{eig:eq8}
\vec A_{m+1}=T^{(m)}\cdots T^{(1)}\vec A_1.
\end{equation}
Note that all the matrices $T^{(i)}$ are $\l$-dependent in a non-linear way.

To satisfy the hinged boundary conditions at $x=0$, we require $\vec{A}_1=[A_1,0,C_1,0]^T$. 

Let the $2\times 4$ matrix $B^{(m+1)}$ transform the solution $w_{m+1}^\l$ determined by the vector $\vec A_{m+1}$ into the boundary conditions for $w_{m+1}^\l$ and $(w_{m+1}^\l)''$ at $x=\pi$.

To satisfy the hinged boundary conditions at $x=\pi$, we have to solve the matrix equation
\begin{equation}\label{eig:eq10}
[0,0]^T=B^{(m+1)}T^{(m)}\cdots T^{(1)}\vec A_1.
\end{equation}

This matrix equation has a non-trivial solution, if the corresponding $\l$-dependent $2\times 2$ determinant is equal to zero. This amounts to finding an eigenvalue  $\l^4$  of the problem. Numerically, the highly nonlinear equation is solved by a Newton type method. The computations can be quite expensive, so the applicability of this method is usually restricted to a small number of cracks, see \cite{LIN2002987}.

\vskip2mm
{\bf Modified Shifrin's method}. The original method is described in \cite{SHIFRIN1999}. We modify it by placing it within the framework of this paper. Also, notice that in our study we have proved the existence of the eigenvalues and the eigenfunctions for the operator $\A$, which provides the theoretical justification for the method.

Let $V_l\subset V$ be the linear space of continuous piecewise linear functions on $[0,\pi]$, which are linear on every interval $l_i=(x_{i-1}, x_i)$, $i=1,\dots m+1$. Note that $V_l$ is an $m$-dimensional space. 

The goal of the next  result is to show that any function $u\in V$ can be uniquely represented as $u=u_s+u_l$, where  $u_s$ is smooth, and $u_l\in V_l$. Thus $u_l$  absorbs all the jumps of the derivative  $u'$ on $(0,\pi)$.
\begin{lemma}\label{eig:lemma2}
Let $u\in V$. Then there exists a unique decomposition
\begin{equation}\label{eig:lemma2:eq2}
u=u_s+u_l,
\end{equation}
where $u_s\in H_0^1(0,\pi)\cap H^2(0,\pi)$, and $u_l\in V_l$.
\end{lemma}
\begin{proof}
Let $v$ be defined by
\[
v(x)=\int_0^x\left[\int_0^\xi u''(\tau)\,d\tau  \right] d\xi, \quad x,\xi\in [0,\pi].
\]
Then $v\in H^2(0,\pi)$, $v(0)=v'(0)=0$, and $v''=u''$.
Let $u_s(x)=v(x)-v(\pi)x/\pi$, $x\in[0,\pi]$. Then $u_s(0)=u_s(\pi)=0$. Therefore $u_s\in H_0^1(0,\pi)\cap H^2(0,\pi)\subset V$. Note that $J[u_s'](x_i)=J[v'](x_i)=0$, and $u_s''=u''$.
If $\tilde u_s$ is another such function, then $\|\tilde u_s -u_s\|_V=0$, and $\tilde u_s=u_s$.

Now we just let $u_l=u-u_s$. Then $u_l''=0$ on every interval $l_i$, and $u_l\in V_l$. The function $u_l$ is unique, since $u-u_l$ is smooth, which is already shown to be unique.
\end{proof}

Let $\varphi=\varphi_k$ be an eigenfunction of $\A$. The index $k$ will be suppressed for notational simplicity.

By Lemma \ref{var:lemma20}{\it(iii)},
\begin{equation}\label{eig:eq22}
\varphi''''(x)=\lambda^4\varphi(x)
\end{equation}
where the equality is satisfied a.e. on every interval $l_i$.
By Lemma \ref{eig:lemma2} we can represent $\varphi$ as
\begin{equation}\label{eig:eq23}
\varphi=\varphi_s+\varphi_l,
\end{equation}
where $\varphi_s$ is smooth, and $\varphi_l\in V_l$. Since $\varphi_l''=0$ on every $l_i$, equation \eqref{eig:eq22} becomes
\begin{equation}\label{eig:eq24}
\varphi_s''''=\lambda^4\varphi_s +\lambda^4\varphi_l.
\end{equation}

Let $\{w_i\}_{i=1}^m$ be the basis in $V_l$, defined by 
\begin{equation}\label{eig:eq26}
w_i(x)=\begin{cases}
\frac{x_i-\pi}{\pi}x, & 0\leq x\leq x_i\\
\frac {x_i}{\pi} (x-\pi), & x_i< x\leq \pi.
\end{cases}
\end{equation} 
Note that $w_i(x)<0$, for $0<x<\pi$,  $J[w_i'](x_i)=1$, and $x_i$ is the only discontinuity point of $w_i'$ on $(0,\pi)$.

Let $\Delta_i=J[\varphi'](x_i)$. Note that $\Delta_i=\Delta_i(\l)$.
Since $J[\varphi_l'](x_i)=J[\varphi'](x_i)$, and $\{w_i\}_{i=1}^m$ is a basis in $V_l$, the function $\varphi_l$ can be represented as
\begin{equation}\label{eig:eq28}
\varphi_l(x;\l)=\sum_{i=1}^m \Delta_i(\l) w_i(x).
\end{equation}

Therefore equation \eqref{eig:eq24} can be written as
\begin{equation}\label{eig:eq30}
\varphi_s''''=\lambda^4\varphi_s +\lambda^4\sum_{i=1}^m \Delta_i w_i.
\end{equation}

In \cite{SHIFRIN1999},  equation (6) on p. 412 is used instead of \eqref{eig:eq24}. Note, that while our approaches are similar, the notations may be defined differently. Thus the formulas in \cite{SHIFRIN1999} cannot be used in our framework.

Equation \eqref{eig:eq30} for $\varphi_s$ is  a linear non-homogeneous, fourth order ODE on $(0,\pi)$. Its general solution 
is the sum of the complementary solution
\begin{equation}\label{eig:eq40}
\left(\varphi_s\right)_c(x;\l)=A\cos(\l x)+B\sin(\l x)+C\cosh(\l x)+D\sinh(\l x),
\end{equation}
and a particular solution $\left(\varphi_s \right)_p(x;\l)$. The latter can be found using the Laplace transform. Its expression in a convolution form is
\begin{equation}\label{eig:eq42}
\left(\varphi_s \right)_p(x;\l)=\l^4\frac 1{2\lambda^3}\left(\sinh(\l x)-\sin(\l x)\right) * \sum_{i=1}^m \Delta_i w_i(x).
 \end{equation}
Let
\begin{equation}\label{eig:eq44}
M_i(x;\l)=\int_0^x\left(\sinh(\l (x-s))-\sin(\l (x-s))\right)  w_i(s)\, ds.
\end{equation}
Then the general solution $\varphi_s=(\varphi_s)_c+(\varphi_s)_p$ of \eqref{eig:eq30} can be written as
\begin{equation}\label{eig:eq46}
\varphi_s(x;\l)=  \left(\varphi_s\right)_c(x;\l)+\frac{\l}{2}  \sum_{i=1}^m \Delta_i   M_i(x;\l).
\end{equation}
Note that all the expressions $ M_i(x;\l), \, i=1,\dots,m$ can be computed explicitly.
 
Recall from the junction conditions \eqref{var:eq4}, that $\varphi''(x_j^-)=\varphi''(x_j^+)$, and $\Delta_j=J[\varphi'](x_j)=\theta_j \varphi''(x_j^+)$ for $j=1,\dots,m$.
By construction $\varphi''(x)=\varphi_s''(x)$. Thus $\varphi''_s(x_j^-)=\varphi''_s(x_j^+)$, and we can write $\Delta_j=\theta_j \varphi''_s(x_j)$, $j=1,\dots,m$. Therefore, the last expression becomes
\begin{equation}\label{eig:eq50}
\Delta_j=\theta_j\left[\left(\varphi_s\right)_c''(x_j;\l) +\frac{\l}{2}  \sum_{i=1}^m \Delta_i   M_i''(x_j;\l) \right], 
\end{equation}
where the derivatives are in $x$, and $j=1,\dots,m$.

Thus \eqref{eig:eq50} gives $m$ linear equations for $m+4$ unknowns $\Delta_j,\,j=1,\dots,m$, and $A, B, C, D$. Note that these $m$ equations are valid for any boundary conditions at the ends of the interval $(0,\pi)$. The missing four equations are derived from the boundary conditions of the problem.

For example, for the hinged boundary conditions \eqref{intro:eq32}, we have the following four linear equations:
$\varphi_s(0;\l)=\varphi_s''(0;\l)=0$, and $\varphi_s(\pi;\l)=\varphi_s''(\pi;\l)=0$. More explicitly, using \eqref{eig:eq46}, equation $\varphi_s(0;\l)=0$ becomes
\begin{equation}\label{eig:eq54}
\left(\varphi_s\right)_c(0;\l)+\frac{\l}{2}  \sum_{i=1}^m \Delta_i   M_i(0;\l)=0,\quad \text{or}\quad A+C=0.
\end{equation}
Similarly, $\varphi_s''(0;\l)=0$ becomes
\begin{equation}\label{eig:eq56}
-\l^2 A+\l^2 C+\frac{\l}{2}  \sum_{i=1}^m \Delta_i   M_i''(0;\l)=0.
\end{equation}
Equations $\varphi_s(\pi;\l)=\varphi_s''(\pi;\l)=0$ are 
\begin{equation}\label{eig:eq58}
\left(\varphi_s\right)_c(\pi;\l)+\frac{\l}{2} \sum_{i=1}^m \Delta_i M_i(\pi;\l)=0,
\end{equation}
and
\begin{equation}\label{eig:eq60}
\left(\varphi_s\right)_c''(\pi;\l)+\frac{\l}{2} \sum_{i=1}^m \Delta_i   M_i''(\pi;\l)=0.
\end{equation}

Writing the linear system \eqref{eig:eq50}--\eqref{eig:eq60} in the matrix form $\bf U(\l) \bar x=\bar 0$ shows that it has non-trivial solutions only if 
\begin{equation}\label{eig:eq62}
\det({\bf U(\l)})=0.
\end{equation}
This non-linear equation has infinitely many solutions $\l_k,\, k\geq 1$, corresponding to the eigenvalues $\l_k^4,\, k\geq 1$. 

Let the corresponding non-trivial solution of the linear system be ${\bf \bar{x}}=\{ \Delta_1(\l_k),\dots, \Delta_m(\l_k), A_k, B_k, C_k, D_k\}$. Then we can compute \newline
$\varphi_l(x;\l_k)=\sum_{i=1}^m \Delta_i(\l_k) w_i(x)$, and
$\varphi_s(x, \l_k)$ from \eqref{eig:eq46}. Finally, we get the eigenfunction $\varphi_k(x)=  \varphi_s(x, \l_k)+\varphi_l(x;\l_k)$. One may want to normalize $\varphi_k$ in $H$ to achieve its uniqueness, up to a sign.

A comparison of the computational efficiency of the methods was conducted in \cite{SHIFRIN1999}. The first three eigenvalues have been computed by the Transition matrices method, and by the Shifrin's method. The Shifrin's method is about twice as fast for the beam with one crack, and about three times as fast for the beam with two cracks.

{\bf Natural beam frequencies}.
For practical applications it is important to express equation \eqref{eig:eq22} in physical variables.

Equation of harmonic transverse oscillations $v=v(x)$ of a uniform beam defined on interval $(0,L)$ is
\begin{equation}\label{phys:eq100}
EIv''''(x)=\omega^2\rho A v(x),\quad 0<x<L,
\end{equation}
where $E$ is the Young's modulus, $A$ and $I$ are the cross-sectional area and the area moment of inertia correspondingly, and $\omega$ is the natural frequency of the oscillations.
For a cracked beam, equation \eqref{phys:eq100} is satisfied on every subinterval $(x_{i-1}, x_i)$, $i=1,\dots, m+1$. 

To relate this equation to the non-dimensional variables, define the $t$-scale $\omega_0$, and the radius of gyration $r$ by
\begin{equation}\label{phys:eq52}
\omega_0=\left(\frac \pi L \right)^2\sqrt{\frac{EI}{\rho A}},\quad r=\sqrt{\frac{I}{A}}. 
\end{equation}
Then make the change of variables
\begin{equation}\label{phys:eq56}
x\leftarrow\frac{\pi x}{L}, \quad v\leftarrow\frac{v}{r},\quad 
t\leftarrow \omega_0 t.
\end{equation}

Then equation \eqref{phys:eq100} becomes
\[
EI r\left(\frac{\pi}{L}\right)^4 v_n''''(x_n, t_n)=\omega^2\rho A r v_n(x_n, t_n).
\]
Thus, in non-dimensional ratios
\[
v''''=\omega^2\left(\frac{L}{\pi}\right)^4\frac{\rho A}{EI}v.
\]
Comparing this equation with the definition of the eigenvalues and the eigenfunctions $\varphi_k''''=\l_k^4\varphi_k$, we conclude that the natural beam frequencies are given by
\begin{equation}\label{phys:eq110}
\omega_k=\l_k^2\left(\frac{\pi}{L}\right)^2\sqrt{\frac{EI}{\rho A}}, \quad k\geq 1.
\end{equation}

{\bf Expressions for flexibilities $\theta_i$}.
The standard approach to modeling a crack is to represent it as a massless rotational spring with the spring constant $k$, and the flexibility $\theta$. 

The spring constant $k$ relates the torque to the angle of rotation. In our case this relationship takes the form $EI y''(x)=k J[v'](x)$, or $J[v'](x)=\theta v''(x)$, where  
\begin{equation}\label{phys:eq22}
\theta=\frac{EI}{k}.
\end{equation}

If the beam has a rectangular cross-section, as shown in Figure \ref{figCracks}, then the area moment of inertia $I$ of the rectangle can be computed explicitly, and \eqref{phys:eq22} can be simplified further. If the crack is double-sided, then by \cite[Eq. (2.8)-(2.10)]{OSTACHOWICZ1991191}, the expression for the flexibility $\theta$ becomes
\begin{equation}\label{phys:eq24}
\theta=6\pi H\hat\mu^2 (0.535-0.929\hat\mu+3.500\hat\mu^2-3.181\hat\mu^3+5.793\hat\mu^4),
\end{equation}
where $H$ is the half-height of the beam cross-section, and $\hat\mu=a/H$.

If the crack is single-sided, then by \cite[Eq. (2.8)-(2.10)]{OSTACHOWICZ1991191}
\begin{equation}\label{phys:eq26}
\theta=6\pi H\hat\mu^2 (0.6384-1.035\hat\mu+3.7201\hat\mu^2-5.1773\hat\mu^3+7.553\hat\mu^4-7.332\hat\mu^5),
\end{equation}
where $H$ is the entire height of the beam cross-section, and $\hat\mu=a/H$.


%

\end{document}